\documentclass[11pt]{amsart}
\usepackage{graphicx}
\usepackage[mathcal]{euscript}

\newtheorem{theorem}{Theorem}[section]
\newtheorem{lemma}[theorem]{Lemma}

\newtheorem{proposition}[theorem]{Proposition}

\theoremstyle{definition}
\newtheorem{definition}[theorem]{Definition}

\theoremstyle{remark}

\numberwithin{equation}{section}

\begin{document}

\title{On knot Floer width and Turaev genus}

\author{Adam Lowrance}
\address{Mathematics Department\\
Louisiana State University\\
Baton Rouge, Louisiana} \email{lowrance@math.lsu.edu}

\subjclass{}
\date{}

\begin{abstract}
To each knot $K\subset S^3$ one can associated its knot Floer homology $\widehat{HFK}(K)$, a finitely generated bigraded abelian group. In general, the nonzero ranks of these homology groups lie on a finite number of slope one lines with respect to the bigrading. The width of the homology is, in essence, the largest horizontal distance between two such lines. Also, for each diagram $D$ of $K$ there is an associated Turaev surface, and the Turaev genus is the minimum genus of all Turaev surfaces for $K$. We show that the width of knot Floer homology is bounded by Turaev genus plus one. Skein relations for genus of the Turaev surface and width of a complex that generates knot Floer homology are given.
\end{abstract}

\maketitle

\section{Introduction}
\label{introduction} Knot Floer homology is an invariant introduced by Ozsv\'ath and Szab\'o (cf. \cite{oz}) and independently by Rasmussen (cf. \cite{ras}) that associates to each knot $K\subset S^3$ a bigraded abelian group $\widehat{HFK}(K)$ whose graded Euler characteristic is the symmetric Alexander polynomial of $K$. The boundary map of the complex that generates knot Floer homology involves counting holomorphic disks in the symmetric product of a Riemann surface. Manolescu, Ozsv\'ath, and Sarkar proved that the boundary map has a combinatorial description (cf. \cite{comb}); however,  knot Floer homology can still be challenging to compute for knots with many crossings.

Since $\widehat{HFK}(K)$ is a finitely generated bigraded abelian group, it is nontrivial in only finitely many bigradings. These nontrivial groups arise on a finite number of slope one lines with respect to the bigrading. Knot Floer width $w_{HF}(K)$ is the largest difference between the $y$-intercepts of two lines that support $\widehat{HFK}(K)$ plus one. 

Each knot diagram has an associated Turaev surface, an unknotted oriented surface on which the knot has an alternating projection. The Turaev genus of a knot $g_T(K)$ is the minimum genus over all Turaev surfaces for the knot. A precise description of the Turaev surface is given in Section \ref{ribsec}. Originally, the Turaev surface was developed to answer questions about the Jones polynomial (cf. \cite{turvpaper}). The genus of this surface gives a natural bound for width of the reduced Khovanov homology $w_{Kh}(K)$, a homology theory for knots whose graded Euler characteristic gives the Jones polynomial (cf. \cite{man}, see also \cite{stoltz}). Manturov showed that $w_{Kh}(K)\leq g_T(K)+1$. The main result of this paper proves an analogous theorem for knot Floer homology. This is the first known application of the Turaev surface in knot Floer homology, and one wonders if there could be more.
\begin{theorem}
\label{main} Let $K\subset S^3$ be a knot. The knot Floer width of $K$ is bounded by the Turaev genus of $K$ plus one: $$w_{HF}(K)\leq g_T(K)+1.$$
\end{theorem}

In \cite{alt}, Ozsv\'ath and Szab\'o classify knot Floer homology for alternating knots. This classification has a nice consequence: if $K$ is an alternating knot, then $w_{HF}(K)=1$. It is also known that $K$ is alternating if and only if $g_T(K)=0$ (cf. \cite{das}). The main result is a generalization of these two facts. This also establishes a common bound for knot Floer and reduced Khovanov width. In general, how these two quantities compare to one another is unknown. Using the tables in \cite{bald} and Bar-Natan's knot atlas, one can compute knot Floer width $w_{HF}(K)$ and reduced Khovanov width $w_{Kh}(K)$ for knots with small crossings, which results in the following observation: if $K$ is a knot with less than 12 crossings, then $w_{HF}(K)=w_{Kh}(K)$. In a recent paper, Manolescu and Ozsv\'ath show that for quasi-alternating links, both $w_{Kh}(K)$ and $w_{HF}(K)$ are equal to one (cf. \cite{mano}).

This paper is organized as follows. Section \ref{widsec} describes the width of the knot Floer complex and how width behaves under a crossing change. Section \ref{ribsec} defines the Turaev surface and describes an algorithm for computing the genus of this surface. Moreover, the behavior of this genus under a crossing change is given. Section \ref{mainsec} gives the proof of the main result and describes an example. Section \ref{skeinsec} gives skein relations for the genus of the Turaev surface and for width of a complex that generates knot Floer homology.\bigskip

Special thanks is given to Scott Baldridge; his guidance has been instrumental in the completion of this paper. The author also thanks Oliver Dasbach, Neal Stoltzfus, and Brendan Owens for many helpful conversations.

\section{Knot Floer Width}
\label{widsec}
\subsection{Kauffman States and the Knot Floer Complex}

The chain complex used to generate knot Floer homology can be described using Kauffman states. These states and their relations to checkerboard graphs are discussed thoroughly by Kauffman (cf. \cite{kauf}, see also \cite{gilmer}). The set of Kauffman states generate (as $\mathbb{Z}$-modules) the chain complex that yields knot Floer homology (cf. \cite{alt}).

Let $D$ be the diagram of an oriented knot $K\subset S^3$ and $\Gamma$ be the 4-valent graph embedded in the plane obtained by changing each crossing of $D$ to a vertex. Choose a marked edge $\varepsilon$ in $\Gamma$, and let $Q$ and $R$ be the two faces of $\Gamma$ that are incident to $\varepsilon$. At each crossing, there are four local faces (not necessarily distinct). A {\it Kauffman state} for $(D,\varepsilon)$ is a map that assigns to each vertex of $\Gamma$ one of the four local faces such that each face of $\Gamma$ except $Q$ and $R$ is assigned to exactly one crossing. This assignment is indicated by placing a dot in one of the four local faces at each crossing, as in Figure \ref{kstate}.

Let $S=S(D,\varepsilon)$ be the set of all Kauffman states for the diagram $D$ with marked edge $\varepsilon$. Define two functions $A:S\to \mathbb{Z}$ and $M:S\to \mathbb{Z}$, called the Alexander filtration level and the Maslov grading respectively. For each vertex in $\Gamma$, the choice of a local face determines the local contribution to both the Maslov and the Alexander gradings as shown in Figures \ref{alex} and \ref{maslov}.
\begin{figure}[h]
\includegraphics[scale=0.7]{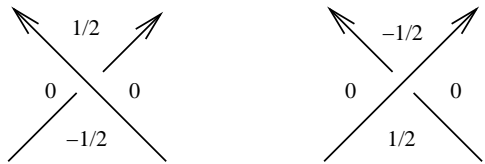}
\caption{The local Alexander filtration level} \label{alex}
\end{figure}
\begin{figure}[h]
\includegraphics[scale=0.7]{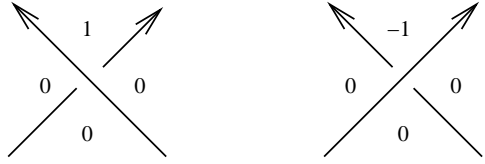}
\caption{The local Maslov grading} \label{maslov}
\end{figure}
The Maslov grading is defined to be the sum of all local Maslov contributions, and the Alexander filtration level is the sum of all local Alexander contributions.

\begin{figure}[h]
\includegraphics[scale=0.5]{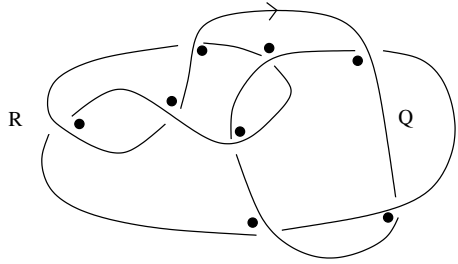}
\caption{This Kauffman state $s$ for the $8_{19}$ knot has $A(s)=-2$ and $M(s)=-5$.} \label{kstate}
\end{figure}

The Kauffman states of a knot diagram $(D,\varepsilon)$ are supported on a finite number of Alexander-Maslov bigrading diagonals. Let 
$$\Delta=\text{max}\{A(s)-M(s) | s\in S\},$$
and 
$$\delta=\text{min}\{A(s)-M(s) | s\in S\}.$$
\begin{definition}
The {\it  width} $w(D,\varepsilon)$ of diagram $D$ of an oriented knot $K$ with distinguished edge $\varepsilon$ is defined to be $w(D,\varepsilon)=\Delta-\delta+1$.
\end{definition}

\subsection{Checkerboard Graphs and a Graph Theoretic Interpretation of Kauffman States}

Recall that there is a graph-theoretic interpretation of Kauffman states (cf. \cite{kauf}). Color the faces of $\Gamma$ black or white in a checkerboard fashion, following the rule that no two faces that share an edge are colored the same. This gives rise to two graphs $T_1$ and $T_2$. The vertices of $T_1$ correspond to the black faces, and the edges of $T_1$ connect vertices whose corresponding black faces are incident to a common vertex in $\Gamma$. Similarly, the vertices of $T_2$ correspond to white faces, and the edges of $T_2$ connect vertices whose corresponding white faces are incident to a common vertex in $\Gamma$. Moreover, label each edge of $T_1$ and $T_2$ by one of the following: $\alpha_+,\alpha_-,\beta_-$ or $\beta_+$, as shown in Figure \ref{edgemark}. 

\begin{figure}[h]
\includegraphics[scale=0.7]{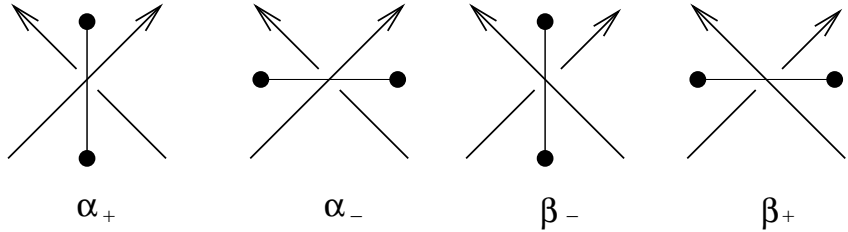}
\caption{The edges in the checkerboard graphs are labeled accordingly.} \label{edgemark}
\end{figure}

Observe that the two checkerboard graphs are embedded in the plane and are dual to one another. Furthermore, note that $\alpha_+$ edges are dual to $\alpha_-$ edges, and $\beta_+$ edges are dual to $\beta_-$ edges. An example of the $8_{19}$ knot and its corresponding checkerboard graphs is given in Figures \ref{shade} and \ref{check}.

\begin{figure}[h]
\includegraphics[scale=.3]{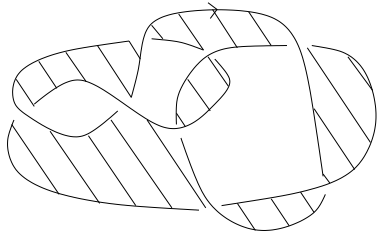}
\caption{The knot $8_{19}$ with a checkerboard coloring} \label{shade}
\end{figure}

\begin{figure}[h]
\includegraphics[scale=.4]{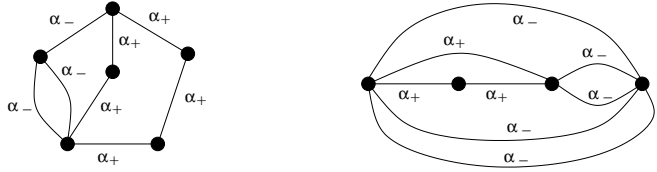}
\caption{The checkerboard graphs for the diagram of $8_{19}$ given in Figure \ref{shade}} \label{check}
\end{figure}

The Kauffman states $S(D,\varepsilon)$ are in 1-1 correspondence with pairs of spanning trees $t_1\subset T_1$ and $t_2\subset T_2$ satisfying the condition that each vertex of $\Gamma$ has exactly one associated edge in $t_1$ or $t_2$ (cf. \cite{kauf}). Observe that this implies if two edges $x$ in $T_1$ and $y$ in $T_2$ are dual to one another, then either $x$ is in $t_1$ or $y$ is in $t_2$. In this paper, a Kauffman state $s\in S(D,\varepsilon)$ will often be identified with a pair of such spanning trees and written as $s=(t_1,t_2)$. The roots of these spanning trees are the vertices corresponding to the two faces $Q$ and $R$ that are incident to $\varepsilon$. Make $t_1$ and $t_2$ into directed graphs by choosing the head of each edge to point away from the root (see Figure \ref{span}).
\begin{figure}[h]
\includegraphics[scale=.5]{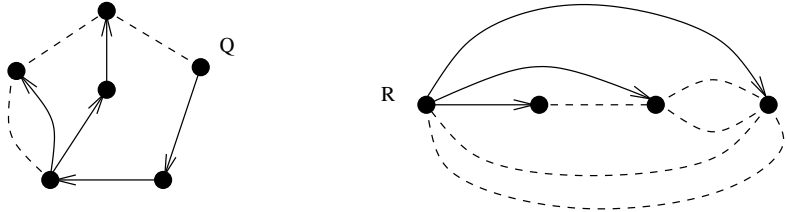}
\caption{The two spanning trees $t_1$ and $t_2$ for the state shown in Figure \ref{kstate}. Solid arcs indicate edges that are in the spanning tree. Dashed arcs indicate edges in the checkerboard graphs but not in the spanning tree.} \label{span}
\end{figure}

\subsection{Width of a Diagram}
The Kauffman state complex depends on the choice of marked edge, however, the width does not.

\begin{proposition}
\label{widind} Let $D$ be an oriented knot diagram, and let $\varepsilon$ and $\varepsilon^\prime$ be marked edges in $\Gamma$. Then  $w(D,\varepsilon)=w(D,\varepsilon^\prime)$.
\end{proposition}
\begin{proof}
Consider the Kauffman state $s=(t_1,t_2)$ as a pair of rooted spanning trees in $T_1$ and $T_2$. The dots of the state $s$ can be recovered as follows. Let $x$ be a directed edge in either $t_1$ or $t_2$. Then $x$ has an associated crossing $c$ in $D$ and the head and tail of $x$ lie in two different local faces around $c$. The local face of $c$ that contains the dot for $s$ is the face that contains the head of $x$. Changing the marked edge in $D$ corresponds to (possibly) changing the root in $t_1$ or $t_2$. This implies that the direction of the edge $x$ may change. However, the local difference between the Alexander and Maslov grading does not depend on the endpoint of $x$ chosen. Notice if $x$ is marked $\alpha_+$, then the local difference is $\frac{1}{2}$, regardless of where the head (ie. the dot in the Kauffman state) is. Similarly, if $x$ is marked $\beta_-$, then the local difference is $-\frac{1}{2}$, and if $x$ is marked $\alpha_-$ or $\beta_+$, then the local difference is $0$ (see Figures \ref{alex}, \ref{maslov}, and \ref{edgemark}). Thus the overall difference $A(s)-M(s)$ remains unchanged, and hence $w(D,\varepsilon)=w(D,\varepsilon^\prime)$. 
\end{proof}

Note that a marked edge is no longer required to define width. Thus width for an oriented knot projection will now be denoted $w(D)$.

The Kauffman state complex is a chain complex that generates knot Floer homology $\widehat{HFK}(K)$ (cf. \cite{alt}). Thus $\widehat{HFK}(K)$ inherits its bigrading from the construction described previously, and there is an analogous notion of width. Let
$$\Delta_K=\text{max}\{A(\xi)-M(\xi)|\xi \text{ is a generator of }\widehat{HFK}(K)\},$$
and 
$$\delta_K=\text{min}\{A(\xi)-M(\xi)|\xi \text{ is a generator of }\widehat{HFK}(K)\}.$$
\begin{definition}
The knot Floer width $w_{HF}(K)$ of a knot $K$ in $S^3$ is given by 
$$w_{HF}(K)=\Delta_K-\delta_K+1.$$
\end{definition}
This immediately implies
\begin{equation}
\label{homineq} w_{HF}(K)\leq\textrm{min}\{w(D)|D \textrm{ is a diagram for } K\}.
\end{equation}

A Kauffman state $s\in S(D,\varepsilon)$ is said to be on the {\it maximal diagonal} if $A(s)-M(s)=\Delta$ and on the {\it minimal diagonal} if $A(s)-M(s)=\delta$. Define a map $\eta:S\to\mathbb{Z}$ by  setting $\eta(s)$ equal to the difference between the number of $\alpha_+$ edges  and $\beta_-$ edges in $s$.  From the proof of Proposition \ref{widind}, the calculation of the local difference for each edge implies that 
$$A(s)-M(s)=\frac{1}{2}\eta(s).$$
Therefore $s$ is on the maximal diagonal if $\eta(s)$ is maximized and on the minimal diagonal if $\eta(s)$ is minimized. Moreover,
$$w(D)=\frac{1}{2}(\text{max}\{\eta(s)|s\in S\}-\text{min}\{\eta(s)|s\in S\})+1.$$

An edge $e$ in either of the checkerboard graphs $T_1$ or $T_2$ is said to be {\it positive} if it is marked either $\alpha_+$ or $\beta_+$; the edge $e$ is said to be {\it negative} if it is marked either $\beta_-$ or $\alpha_-$. If $e$ is contained in a cycle consisting of only positive edges, then $e$ is said to be in a {\it positive cycle}. If $e$ is contained in a cycle consisting of only negative edges, then $e$ is said to be in a {\it negative cycle}. 

The width of a diagram behaves predictably under a crossing change. Before this behavior can be described, a lemma is needed. 

\begin{lemma}
\label{cycle} Let $D$ be a diagram with marked edge $\varepsilon$ for the knot $K$, and let $e$ be an edge in either of the checkerboard graphs $T_1$ or $T_2$.
\begin{enumerate}
\item If $e$ is in a positive (negative) cycle, then there exists a state $s\in S(D,\varepsilon)$ on the maximal (minimal) diagonal that does not contain $e$.
\item If $e$ is a positive (negative) edge and is not in a positive (negative) cycle, then every state $s\in S(D,\varepsilon)$ on the maximal (minimal) diagonal must contain $e$. 
\end{enumerate}
\end{lemma}
\begin{proof}
Only the statements for positive edges are proved; the proofs for the negative edges are analogous. Without loss of generality, suppose that $e$ is a positive edge in $T_1$.\bigskip

\noindent \textbf{(1)} Suppose $e$ is in a positive cycle $\gamma$ and suppose all states on the maximal diagonal contain $e$. Let $s$ be a state on the maximal diagonal consisting of the two spanning trees $t_1\subset T_1$ and $t_2\subset T_2$. Since $t_1$ contains the edge $e$, there exists some other edge $e^\prime$ in $\gamma$ not contained in $t_1$. The graph obtained by adding the edge $e^\prime$ to $t_1$ contains a unique cycle $\tau$. 

Suppose the edge $e$ is contained in this unique cycle. Then form a new state $s^\prime$ consisting of two new spanning trees $t_1^\prime$ and $t_2^\prime$, where $t_1^\prime$ is the spanning tree obtained by adding $e^\prime$ and deleting $e$ in $t_1$, and $t_2^\prime$ is the spanning tree obtained by deleting the dual of $e^\prime$ and adding the dual of $e$ in $t_2$.

To show that $s^\prime$ is on the maximal diagonal, it is enough to show that $\eta(s)=\eta(s^\prime)$. Since $e$ and $e^\prime$ are in a positive cycle, both edges are positive, and the dual edges are negative. Deleting a positive edge from $t_1$ and adding its negative dual to $t_2$ results in a net decrease of $\eta(s)$ by one, since this corresponds to removing an $\alpha_+$ from $t_1$ edge and replacing it with an $\alpha_-$ edge in $t_2$ or removing a $\beta_+$ edge from $t_1$ and replacing it with a $\beta_-$ edge in $t_2$. Likewise, deleting a negative edge from $t_1$ and adding its positive dual to $t_2$ results in a net increase of $\eta(s)$ by one. To construct $s^\prime$, first a positive edge is removed from $t_1$ and its negative dual is inserted into $t_2$. Then a negative edge is removed from $t_2$ and its positive dual is inserted into $t_1$. Thus $\eta(s)=\eta(s^\prime)$, and $s^\prime$, a state not containing the edge $e$, is on the maximal diagonal.

Now suppose the edge $e$ is not contained the cycle $\tau$. Thus $\tau\neq\gamma$, and there is some edge $e^{\prime\prime}$ in $\tau$ not contained in $\gamma$. Construct a new state $s^{\prime\prime}$ by deleting $e^{\prime\prime}$ from $t_1$, replacing it with its dual in $t_2$, inserting $e^\prime$ into $t_1$, and deleting its dual from $t_2$. Notice that if $e^{\prime\prime}$ is a negative edge, then two negative edges were deleted and two positive edges were inserted in the construction of $s^{\prime\prime}$. Thus $\eta(s^{\prime\prime})=\eta(s)+2$, contradicting the fact that $s$ is on the maximal diagonal. Hence $e^{\prime\prime}$ must be a positive edge, and the construction of $s^{\prime\prime}$ simultaneously exchanges a positive for a negative edge and a negative for a positive edge. Therefore $\eta(s^{\prime\prime})=\eta(s)$, and $s^{\prime\prime}$ is again on the maximal diagonal.

Iterate this process as follows: continue by choosing a new edge in $\gamma$ not in $s^{\prime\prime}$ (and thus this edge is also not in $s$). Adding this new edge to $s^{\prime\prime}$ forms a unique cycle. If $e$ is contained in this unique cycle, the process ends as described above. If $e$ is not contained in this unique cycle, then some edge not in $\gamma$ can be removed, resulting in a state still on the maximal diagonal. Since $\gamma$ contains only a finite number of edges, in a finite number of steps, the edge $e$ must be contained in the unique cycle. Therefore there is a state on the maximal diagonal not containing the edge $e$.\bigskip

\noindent\textbf{(2)} Suppose $e$ is a positive edge and is not in a positive cycle. Also, suppose that there exists a state $s$ on the maximal diagonal, consisting of spanning trees $t_1\subset T_1$ and $t_2\subset T_2$, not containing the edge $e$. Then consider the subgraph of $T_1$ obtained by adding the edge $e$. There is a unique cycle in this subgraph, and since $e$ is not in a positive cycle, this cycle contains a negative edge $e^\prime$. Let $s^\prime$ be the state obtained by deleting $e^\prime$ and adding $e$ in $t_1$ and adding the dual of $e^\prime$ and deleting the dual of $e$ in $t_2$. Both of these switches adds a positive edge and deletes a negative edge. Thus $\eta(s^\prime)=\eta(s)+2$ and this contradicts the fact that $s$ is on the maximal diagonal. Hence all states on the maximal diagonal must contain $e$.
\end{proof}\bigskip

With the previous lemma established, the behavior of the width of a diagram under a crossing change can now be determined. Let $D$ be a diagram for the knot $K$ with marked edge $\varepsilon$, and let $D^\prime$ be the diagram obtained from $D$ by a single crossing change. Let $T_1$ and $T_2$ be the checkerboard graphs for $D$ and $T_1^\prime$ and $T_2^\prime$ be the checkerboard graphs for $D^\prime$. The crossing in $D$ has an associated positive edge $e_+$ and an associated negative edge $e_-$ in the checkerboard graphs. These two edges are dual to each other. Moreover, the crossing change switches $e_+$ to a negative edge and $e_-$ to a positive edge.

\begin{theorem}
\label{widthm} Let $D$ be a diagram of a knot $K$ and $D^\prime$ be the diagram obtained from $D$ by a single crossing change. Suppose $e_+$ (the positive edge) and $e_-$ (the negative edge) are the edges in the checkerboard graphs $T_1$ and $T_2$ of $D$ associated to the crossing that is change. Then the width of a diagram under a crossing change behaves as follows:
\begin{enumerate}
\item $|w(D)-w(D^\prime)|\leq 1$.
\item If $e_+$ is in a positive cycle and $e_-$ is in a negative cycle, then $w(D^\prime)=w(D)+1$.
\item If $e_+$ is in a positive cycle and $e_-$ is not in any negative cycle, then
$w(D^\prime)=w(D)$.
\item If $e_+$ is not in any positive cycle and $e_-$ is in a negative cycle, then 
$w(D^\prime)=w(D)$. 
\item If $e_+$ is not in any positive cycle and $e_-$ is not in any negative cycle, then $w(D^\prime)=w(D)-1$.
\end{enumerate}
\end{theorem}
\begin{proof}
\noindent \textbf{(1)} Let $s=(t_1,t_2)$ be a Kauffman state for $D$. The edges $e_+$ and $e_-$ are dual in the checkerboard graphs. Thus exactly one of them is an edge in either $t_1$ or $t_2$. The crossing change corresponds to changing this edge and no others in $t_1$ or $t_2$. Label the new Kauffman state for $D^\prime$ by $s^\prime=(t_1^\prime,t_2^\prime)$.

Suppose $e_+$ is marked $\alpha_+$. Then $e_-$ is marked $\alpha_-$, and after the crossing change, $e_+$ switches to a $\beta_-$ edge and $e_-$ switches to a $\beta_+$ edge (see Figure \ref{edgemark}). Thus if $e_+$ is in either $t_1$ or $t_2$, it follows that $\eta(s)=\eta(s^\prime)+2$ and $A(s)-M(s)=A(s^\prime)-M(s^\prime)+1$. If $e_+$ is not in either $t_1$ or $t_2$, then $e_-$ must be in either $t_1$ or $t_2$. Then $\eta(s)=\eta(s^\prime)$ and $A(s)-M(s)=A(s^\prime)-M(s^\prime)$. This implies that $\Delta$ and $\delta$ either decrease by one or remain the same. The case where $e_+$ is marked $\beta_+$ is analogous. Therefore, $|w(D)-w(D^\prime)|\leq 1$.\bigskip

\noindent \textbf{(2)} Suppose $e_+$ is in a positive cycle and $e_-$ is in a negative cycle. Then by Lemma \ref{cycle}, there are states $s_{\textrm{max}}$ and $s_{\textrm{min}}$ in $S(D,\varepsilon)$ such that $s_{\textrm{max}}$ is on the maximal diagonal and does not contain $e_+$ and $s_{\textrm{min}}$ is on the minimal diagonal and does not contain $e_-$. Since $e_+$ and $e_-$ are dual and $s_{\textrm{max}}$ does not contain $e_+$, it follows that $s_{\textrm{max}}$ contains $e_-$. Similarly, $s_{\textrm{min}}$ contains $e_+$. Let $s_{\textrm{max}}^\prime$ and $s_{\textrm{min}}^\prime$ be the states after the crossing change with the same edges as $s_{\textrm{max}}$ and $s_{\textrm{min}}$ respectively. 

If $e_+$ is marked $\alpha_+$, then $e_-$ is marked $\alpha_-$. After the crossing change, $e_+$ is switched to a $\beta_-$ edge, and $e_-$ is switched to a $\beta_+$ edge. It follows that $\eta(s_{\textrm{max}}^\prime)=\eta(s_{\textrm{max}})+2$ and $\eta(s_{\textrm{min}}^\prime)=\eta(s_{\textrm{min}})$. In this case, the crossing change induces an increase in $\Delta$ by one and no change in $\delta$. Similarly, if $e_+$ is marked $\beta_+$ and $e_-$ is marked $\beta_-$, then the crossing change induces no change in $\Delta$ and a decrease in $\delta$ by one. Therefore, $w(D^\prime)=w(D)+1$.\bigskip

\noindent \textbf{(3)} Suppose $e_+$ is in a positive cycle and $e_-$ is not in any negative cycle. As before, there is a state $s_{\textrm{max}}$ in $S(D,\varepsilon)$ on the maximal diagonal not containing $e_+$. Now, however, every state on the minimal diagonal must contain the edge $e_-$. Hence $s_{\textrm{max}}$, as well as every state on the minimal diagonal contains the edge $e_-$. So, if $e_+$ is marked $\beta_+$ and $e_-$ is marked $\beta_-$, then both $\Delta$ and $\delta$ are increased by one under a crossing change. If $e_+$ is marked $\alpha_+$ and $e_-$ is marked $\alpha_-$, then the crossing change does not alter $\eta(s_{\textrm{max}})$ or $\eta(s)$ for $s$ any state on the minimal diagonal. If another state $s_{\textrm{max}}^\prime$ on the maximal diagonal contains the edge $e_+$, then the crossing change decreases $\eta(s_{\textrm{max}}^\prime)$ by two. Therefore $\Delta$ and $\delta$ are unchanged. Thus $w(D^\prime)=w(D)$.\bigskip

\noindent \textbf{(4)} Suppose $e_+$ is not in any positive cycle and $e_-$ is in a negative cycle. This case is completely analogous to the case above. If $e_+$ is marked $\beta_+$ and $e_-$ is marked $\beta_-$, then both $\Delta$ and $\delta$ remained unchanged, and if $e_+$ is marked $\alpha_+$ and $e_-$ is marked $\alpha_-$, then both $\Delta$ and $\delta$ are decreased by one. Therefore, $w(D^\prime)=w(D)$.\bigskip

\noindent \textbf{(5)} Suppose $e_+$ is not in any positive cycle and that $e_-$ is not in any negative cycle. Then all states on the maximal diagonal contain $e_+$ and all states on the minimal diagonal contain $e_-$. If $e_+$ is marked $\beta_+$ and $e_-$ is marked $\beta_-$, then $\Delta$ remains unchanged and $\delta$ is increased by one. If $e_+$ is marked $\alpha_+$ and $e_-$ is marked $\alpha_-$, then $\Delta$ is decreased by one and $\delta$ remains unchanged. Thus $w(D^\prime)=w(D)-1$.
\end{proof}

\section{Ribbon Graphs and Turaev genus}
\label{ribsec}

The ideas discussed below involve ribbon graphs associated to a knot diagram. These ideas are developed in a paper by Dasbach, Futer, Kalfagianni, Lin, and Stoltzfus (cf. \cite{das}).

\subsection{Ribbon Graphs}
A {\it connected oriented ribbon graph} $\mathbb{D}$ is a connected graph embedded on an oriented surface such that each face of the graph is homeomorphic to a disk. Informally, we think of a connected oriented ribbon graph as a graph together with a cyclic ordering of the edges around each vertex. The surface on which the graph embeds is then the smallest genus oriented surface in which the graph can be embedded while preserving the cyclic ordering of the edges around each vertex. The {\it genus} $g(\mathbb{D})$ of a connected oriented ribbon graph $\mathbb{D}$ is the genus of the surface on which the graph is embedded and is determined by its Euler characteristic. Note that all ribbon graphs in this paper are assumed to be connected and oriented and are referred to only as ribbon graphs.

For each planar knot diagram, there are two associated ribbon graphs $\mathbb{D}(A)$ and $\mathbb{D}(B)$. Let $D$ be a diagram for a knot $K\subset S^3$. For each crossing in $D$, there is an $A$-splicing and a $B$-splicing,  and each time a crossing is replaced with a splicing, an edge is inserted as shown in Figure \ref{splicing}.

\begin{figure}[h]
\includegraphics[scale=.5]{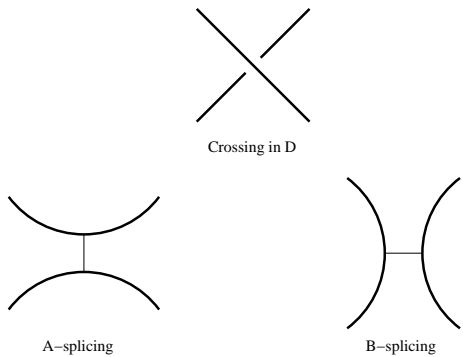}
\caption{The splicings of a crossing} \label{splicing}
\end{figure}

One ribbon graph $\mathbb{D}(A)$ is associated to choosing the $A$-splicing for each crossing, and the other $\mathbb{D}(B)$ is associated to choosing the $B$-splicing for each crossing. The construction of $\mathbb{D}(A)$ is described here; the construction of $\mathbb{D}(B)$ is analogous. First, take a checkerboard coloring of $\Gamma$ as described above. Then draw a circle corresponding to each of the black faces. Connect the circles just as their corresponding faces are connected in $D$ (see the first picture of Figure \ref{ribalg}). Next, replace each crossing by an $A$-splicing (see Figure \ref{splicing}). This results in a collection of circles in the plane together with line segments joining them. Choose an orientation for each circle as follows. Orient the circle counterclockwise if it is inside an even number of circles, and orient the circle clockwise if it is inside an odd number of circles (see the second picture of Figure \ref{ribalg}).

The ribbon graph $\mathbb{D}(A)$ is obtained by ``contracting each circle to a point" as follows: the vertices of $\mathbb{D}(A)$ are in one-to-one correspondence with the circles. Two vertices of $\mathbb{D}(A)$ are connected by an edge if their associated circles have a line segment between them. The cyclic orientation of the edges meeting at a vertex of $\mathbb{D}(A)$ is determined by first fixing a cyclic orientation of the plane, say counterclockwise. Then the edges meeting at any vertex of $\mathbb{D}(A)$ are cyclically ordered in the counterclockwise direction according to the cyclic order given by the orientation of the corresponding circle. Figure \ref{ribalg} describes each step of this construction. Note to construct $\mathbb{D}(B)$, start with circles corresponding to the white faces of $\Gamma$, and at each crossing choose a $B$-splicing instead of an $A$-splicing. Otherwise proceed as above.

In the construction of the two ribbon graphs $\mathbb{D}(A)$ and $\mathbb{D}(B)$, a choice of black and white checkerboard graphs is made. The construction does not depend on this choice. Regardless of whether circles corresponding to the white or black graph are chosen, the circles coming from choosing an $A$-splicing at each crossing (or choosing a $B$-splicing at each crossing) are the same. 

\begin{figure}[h]
\includegraphics[scale=.5]{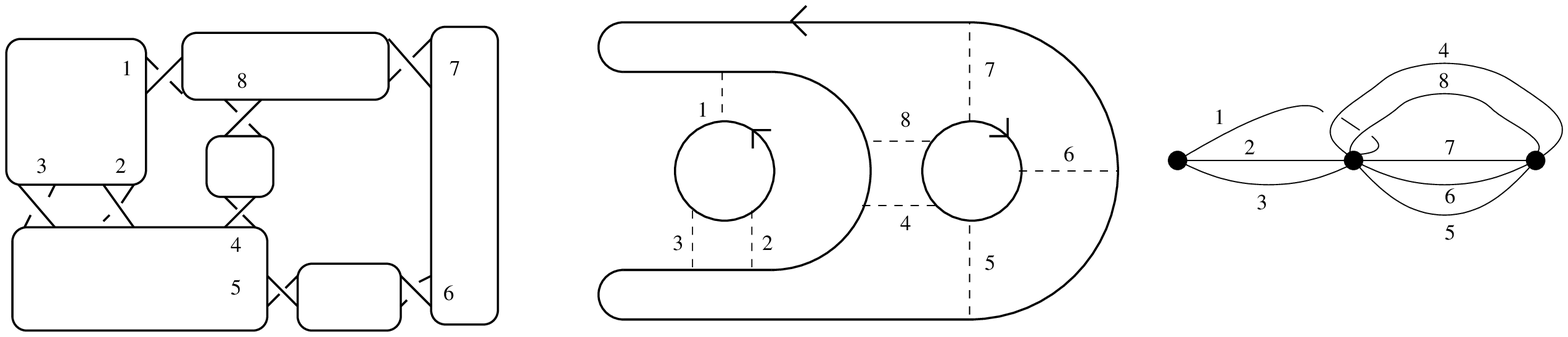}
\caption{On the left is the black faces of $8_{19}$ represented as circles with the crossing information from $D$. The picture in the middle is the result of choosing an $A$-splicing at each crossing. On the right is the ribbon graph $\mathbb{D}(A)$ for $8_{19}$.} \label{ribalg}
\end{figure}

\subsection{Turaev Surface}

The ribbon graph $\mathbb{D}(A)$ is embedded on a surface as follows. Let $D$ be the diagram and $\Gamma$ the plane graph associated to $D$. Regard $\Gamma$ as embedded in $\mathbb{R}^2$ sitting inside $\mathbb{R}^3$. Outside the neighborhoods around the vertices of $\Gamma$ is a collection of arcs in the plane. Replace each arc by a band that is perpendicular to the plane. In the neighborhoods of the vertices, place a saddle so that the circles obtained from choosing an $A$-splicing at each crossing lie above the plane and so that the circles obtained from choosing a $B$-splicing at each crossing lie below the plane (see Figure \ref{saddle}). This results in a surface with boundary a collection of disjoint circles, with circles corresponding to the $A$-splicing above the plane and circles corresponding to the $B$-splicing below the plane. For each boundary circle, insert a disk, to obtain a closed surface $G(D)$ known as the {\it Turaev surface} (cf. \cite{turvpaper}).

\begin{figure}[h]
\includegraphics[scale=.4]{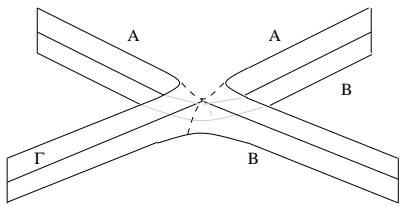}
\caption{In a neighborhood of each vertex of $\Gamma$ a saddle surface transitions between the $A$ and $B$ circles.} \label{saddle}
\end{figure}

The ribbon graph $\mathbb{D}(A)$ is embedded on $G(D)$ as follows. Each vertex of $\mathbb{D}(A)$ is the center of the disk used to cap off a boundary circle lying above the plane. Edges are then gradient lines from the vertices through the saddle points (which  correspond to crossings). Notice that the ribbon graph $\mathbb{D}(B)$ can also be embedded in this surface by embedding its vertices in the center of disks used to cap off circles below the plane. Edges of $\mathbb{D}(B)$ are also gradient lines from the vertices to the saddle points.

The embeddings of $\mathbb{D}(A)$ and $\mathbb{D}(B)$ on $G(D)$ are especially nice. Each face of both $\mathbb{D}(A)$ and $\mathbb{D}(B)$ on $G(D)$ is homeomorphic to a disk. Moreover, $\mathbb{D}(A)$ and $\mathbb{D}(B)$ are dual on $G(D)$ (cf. \cite{das}). Note that since each face of $\mathbb{D}(A)$ and $\mathbb{D}(B)$ is a disk, it follows that the genera of $\mathbb{D}(A),\mathbb{D}(B),$ and $G(D)$ agree. This leads to the definition of the topological invariant discussed in the introduction.

\begin{definition}
Let $K\subset S^3$ be a knot. The {\it Turaev genus} $g_T(K)$ of $K$ is defined by
$$g_T(K)=\textrm{min}\{g(G(D))|D\textrm{ is a diagram of }K\}.$$
\end{definition}

Turaev genus is an obstruction to a knot being alternating. In fact, a knot $K$ is alternating if and only if $g_T(K)=0$ (cf. \cite{das}).

\subsection{Computing the genus of the Turaev surface}
\label{alg}

The genus of $G(D)$ is determined by the Euler characteristic of $\mathbb{D}(A)$ (thought of as a cellular decomposition of $G(D)$). Thus to compute the genus of the Turaev surface, it suffices to compute the number of vertices $V$, the number of edges $E$, and the number of faces $F$ of $\mathbb{D}(A)$. The number of edges $E$ is equal to the number of crossings in the diagram. Since $\mathbb{D}(A)$ and $\mathbb{D}(B)$ are dual to one another on $G(D)$, the number of faces $F$ of $\mathbb{D}(A)$ is equal to the number of vertices of $\mathbb{D}(B)$. 

In order to compute $V$, it suffices to count the number of circles after choosing the $A$-splicing at each crossing in the constructions of $\mathbb{D}(A)$. The vertices of $T_1$ correspond to the circles coming from the black checkerboard coloring of $D$. These circles, along with the crossings, determine the knot diagram, and are the starting point for the construction of $\mathbb{D}(A)$ (see Figure \ref{ribalg}). The following is an algorithm to count the vertices of $\mathbb{D}(A)$ by counting the circles after choosing the $A$-splicing for each crossing. The algorithm is given by performing a sequence of operations on the checkerboard graph $T_1$. 

\begin{figure}[h]
\includegraphics[scale=.4]{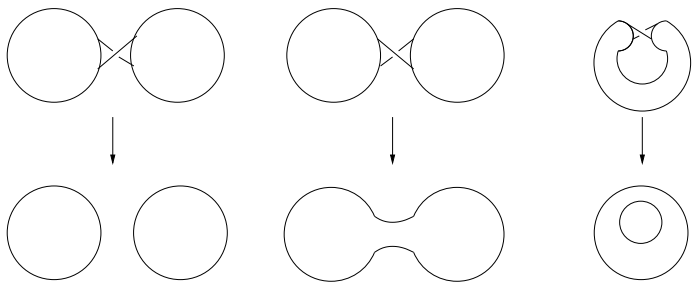}
\caption{On the left is the picture corresponding to deleting a negative edge. The middle is the picture corresponding to contracting a non-loop positive edge. On the right is the picture corresponding choosing the $A$-splicing for a positive loop.} \label{contract}
\end{figure}

{\it Step 1: Remove all negative edges from $T_1$.} If two vertices of $T_1$ are the endpoints of a negative edge, then their corresponding circles are separated by an $A$-splicing (see Figure \ref{contract}). Thus choosing an $A$-splicing for that crossing does not change the number of circles, and so each negative edge in $T_1$ can be removed.

{\it Step 2: Contract all non-loop positive edges.} If in the resulting graph there exists a positive edge whose endpoints are distinct vertices, then the circles corresponding to these vertices are joined by an $A$-splicing (see the second picture of Figure \ref{contract}). Thus choosing an $A$-splicing for that crossing decreases the number circles by one; likewise, contracting the edge decreases the number of vertices by one. Either the resulting graph contains a non-loop positive edge or all remaining edges are loops. If the graph contains a non-loop positive edge, then repeat this step. Otherwise, the resulting graph is a collection of vertices and loops, and is called {\it the bouquet of $T_1$}.

{\it Step 3: Count vertices and loops.} Each vertex in the bouquet of $T_1$ corresponds to a circle in the construction of $\mathbb{D}(A)$, and each loop in the bouquet of $T_1$ corresponds to a crossing between a circle and itself. Choosing an $A$-splicing at a crossing between a circle and itself splits that circle into two circles (see Figure \ref{contract}). Therefore, each loop also corresponds to a circle in the construction of $\mathbb{D}(A)$. Hence, $V$ is the number of vertices plus the number of loops in the bouquet of $T_1$.

In order to calculate $F$, the algorithm is modified as follows. The sequence of operations is performed on the checkerboard graph $T_2$. Since a $B$-splicing is chosen at each crossing in the construction of $\mathbb{D}(B)$, 
in Step 1, positive edges are deleted. Also, non-loop negative edges are contracted in Step 2. Then $F$ is equal to the number of vertices plus the number of loops in {\it the bouquet of $T_2$}. This process is shown in Figure \ref{vcount}.\\
\begin{figure}[h]
\includegraphics[scale=0.4]{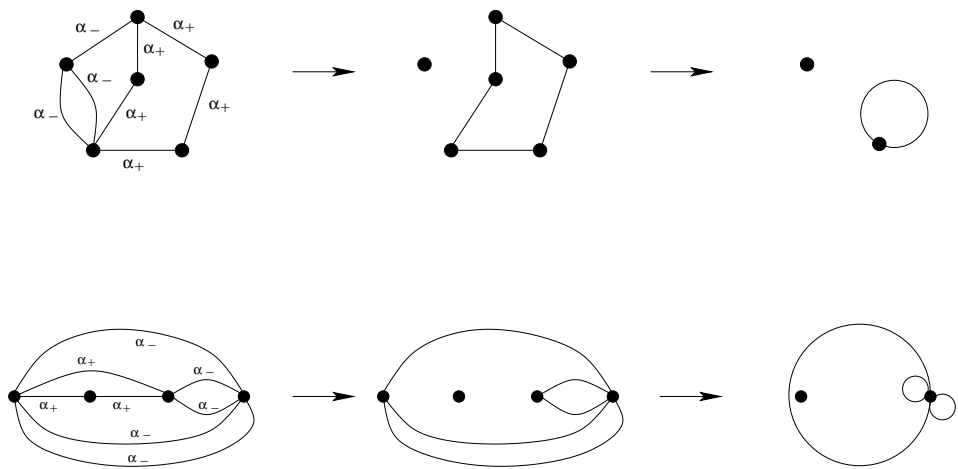}
\caption{The procedure on $T_1$ (top) and $T_2$ (bottom) gives $V=3$ and $F=5$.} \label{vcount}
\end{figure}
This algorithm immediately implies the following theorem.

\begin{theorem}
\label{turgen} Let $D$ be a diagram for a knot $K\subset S^3$, and let $G(D)$ be the Turaev surface of $D$. Let $T_1$ and $T_2$ be the checkerboard graphs of $D$. Let $V$ be the number of vertices and loops in the bouquet of $T_1$, $E$ be the number of edges in $T_1$ (or $T_2$), and $F$ be the number of vertices and loops in the bouquet of $T_2$. Then
$$2-2g(G(D))=V-E+F.$$
\begin{flushright}
$\Box$
\end{flushright}
\end{theorem}

Since the constructions of $\mathbb{D}(A)$ and $\mathbb{D}(B)$ do not depend on which checkerboard graph is chosen, it follows that this algorithm does not depend on the checkerboard graph chosen. Thus $T_1$ and $T_2$ may be relabeled at our convenience.

We next investigate the behavior of the genus of the Turaev surface under a crossing change in light of the algorithm given in the proof of Theorem \ref{turgen}. Let $D$ be a diagram for the knot $K$, and let $D^\prime$ be the diagram obtained from $D$ by a single crossing change. Let $T_1$ and $T_2$ be the checkerboard graphs for $D$. Let $G(D)$ and $G(D^\prime)$ be the two Turaev surfaces. Suppose that $e_+$ and $e_-$ are the edges in the checkerboard graphs that are associated to the crossing that is changed. Assume that $e_+$ is a positive edge. Since $e_-$ is dual to $e_+$, it follows that $e_-$ is negative (see Figure \ref{edgemark}). The crossing change causes $e_+$ to switch to a negative edge and $e_-$ to switch to a positive edge. 

\begin{theorem}
\label{genthm} Let $D$ be a diagram of a knot $K$ and $D^\prime$ be the diagram obtained from $D$ by a single crossing change. Suppose $e_+$ (the positive edge) and $e_-$ (the negative edge) are the edges in the checkerboard graphs $T_1$ and $T_2$ of $D$ associated to the crossing that is changed. Then the genus of the Turaev surface under a crossing change behaves as follows:
\begin{enumerate}
\item If $e_+$ is in a positive cycle and $e_-$ is in a negative cycle, then $g(G(D^\prime))=g(G(D))+1$.
\item If $e_+$ is in a positive cycle and $e_-$ is not in any negative cycle, then $g(G(D^\prime))=g(G(D))$.
\item If $e_+$ is not in any positive cycle and $e_-$ is in a negative cycle, then $g(G(D^\prime))=g(G(D))$.
\item If $e_+$ is not in any positive cycle and $e_-$ is not in any negative cycle, then $g(G(D^\prime))=g(G(D))-1$.
\end{enumerate}
\end{theorem}
\begin{proof}

If $e_+$ is not an edge in $T_1$, then relabel $T_1$ and $T_2$ so that it is. In order to compute $V$, the algorithm of Theorem \ref{turgen} states that all negative edges are removed from $T_1$. Since a crossing change switches $e_+$ to a negative edge, after the crossing change this edge will be deleted. If $e_+$ is in a positive cycle, then this decreases the number of loops in the bouquet of $T_1$ by one, and thus $V$ decreases by one. If $e_+$ is not in any positive cycle, then this increases the number of vertices in the bouquet of $T_1$ by one, and thus $V$ increases by one. 

Similarly, in order to compute $F$ all positive edges are deleted from $T_2$, and after the crossing change, the edge corresponding to $e_-$ will be deleted. If $e_-$ is in a negative cycle, then this decreases the number of loops in the bouquet of $T_2$ by one, and thus $F$ decreases by one. If $e_-$ is not in any negative cycle, this increases the number of vertices in the bouquet of $T_2$ by one, and thus $F$ increases by one.

The number of edges $E$ is equal to the number of crossings in the diagram, which remains the same under a crossing change. These conditions determine the behavior of the Euler characteristic, and thus the genus of $G(D)$ under a crossing change. 
\end{proof}

\section{Knot Floer width and Turaev genus}
\label{mainsec}

In the previous sections, we showed that the behavior of both the width of a diagram and the genus of the Turaev surface under a crossing change mimic each other.

\begin{theorem}
\label{width} Let $D$ be a diagram for a knot $K\subset S^3$ and $G(D)$ be the Turaev surface for $D$. Then $w(D)=g(G(D))+1$.
\end{theorem}

\begin{proof}
Let $D$ be a diagram for $K$. If $D$ is an alternating diagram, then the Kauffman states appear on only one Maslov-Alexander diagonal, and $w(D)=1$ (cf. \cite{alt}). Also, if $D$ is an alternating diagram, then $G(D)$ is a sphere (cf. \cite{das}), and hence the result holds for alternating knots. Since any knot diagram can be obtained from an alternating diagram through a sequence of crossing changes, Theorem \ref{genthm} and Theorem \ref{widthm} imply the result.
\end{proof}

Theorem \ref{width} along with the algorithm of Theorem \ref{turgen} give a method to calculate the width of the Kauffman state complex for a diagram $D$. The main theorem (Theorem \ref{main}) is a direct consequence of the previous theorem and Equation \ref{homineq}.\bigskip

We elaborate on the example of the $8_{19}$ knot developed throughout the paper. Figure \ref{kstate} shows a Kauffman state $s$ where $A(s)-M(s)=3$, and Figure \ref{kstate2} shows a state $s$ where $A(s)-M(s)=2$.
\begin{figure}[h]
\includegraphics[scale=0.5]{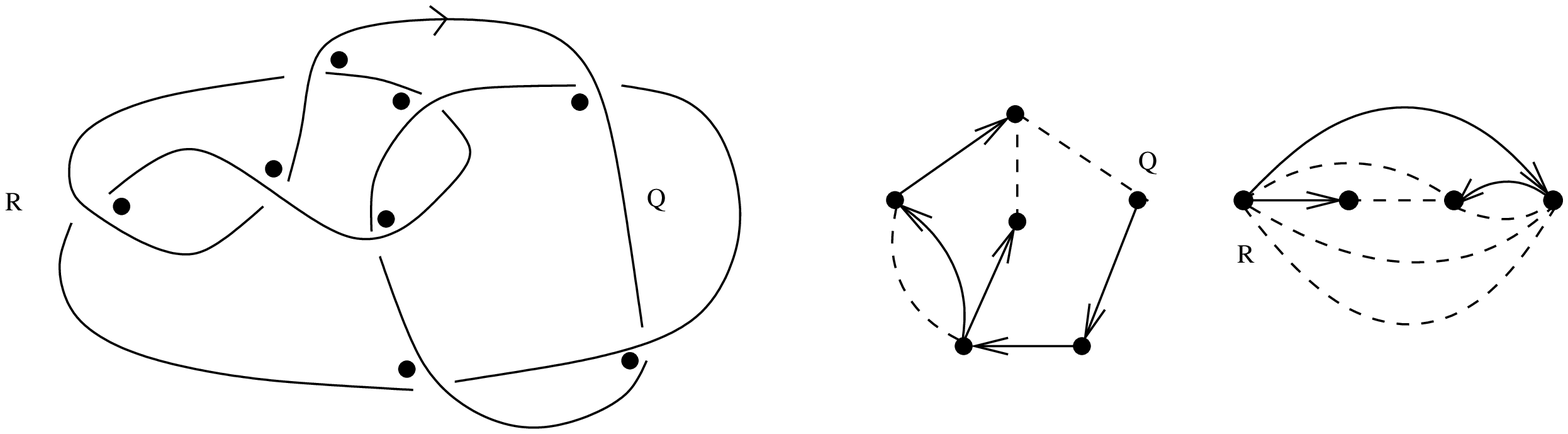}
\caption{For this Kauffman state, $\eta(s)=4$, and hence $A(s)-M(s)=2$.} \label{kstate2}
\end{figure}
In fact, there are 27 Kauffman states for this diagram of $8_{19}$. The number of Kauffman states in each bigrading is listed in Table \ref{tab1}. From this information, one can see that $w(D)=2$. Figure \ref{kstate} shows a state on the maximal diagonal, and Figure \ref{kstate2} shows a state on the minimal diagonal.

\begin{table}[h]
\caption{This shows the number of Kauffman states in each bigrading.}
\begin{tabular}{|c||c|c|c|c|c|c|c|}
\hline
  Alexander$\backslash$Maslov & -6 & -5 & -4 & -3 & -2 & -1 & 0 \\
  \hline\hline
  -3& 1&&&&&&\\
  \hline
  -2& & 2& 1&&&&\\
  \hline
  -1& & & 3 & 3 & & &\\
  \hline
  0&&&&3&4&&\\
  \hline
  1&&&&&3&3&\\
  \hline
  2&&&&&&2&1\\
  \hline
  3&&&&&&&1\\
  \hline
\end{tabular}
\label{tab1}
\end{table}

Figure \ref{vcount} shows that for the given diagram of $8_{19}$, $V=3$ and $F=5$. Since the diagram has $8$ crossings, it follows the $E=8$. Therefore the Euler characteristic of $G(D)$ is zero, and the genus of $G(D)$ is one. This verifies that $w(D)=g(G(D))+1$. In fact, the knot Floer width of $8_{19}$ is two (cf. \cite{bald}), and since $8_{19}$ is non-alternating, its Turaev genus is one.

\section{A skein relation for genus and width}
\label{skeinsec}

In this section skein relations for the genus of the Turaev surface and for the width of a diagram are developed. Each of these skein relations is for a link diagram. Therefore, these relations cannot immediately be used to calculate knot Floer width or Turaev genus. Instead, they give us an upper bound for each.

In order to define the relations, we must expand our view from knots to links. The construction of the Turaev surface can be generalized to include links. If the diagram of the link is non-split (ie. there is no circle in the plane that does not intersect the diagram, where part of the diagram lies both inside and outside the circle), then both $\mathbb{D}(A)$ and $\mathbb{D}(B)$ are connected, and the construction of the Turaev surface is the same as before. However, if the link diagram is split, then $D=D_1\coprod\dots\coprod D_n$ is a disjoint union of non-split diagrams. Each $D_i$ is called a {\it split component of $D$}. The Turaev surface corresponds to a disjoint union of surfaces, one for each split component of the link diagram. 

For a link diagram $D$, let $\chi(G(D))$ be the Euler characteristic of the Turaev surface. If $D = D_1\coprod D_2$ is a split diagram with split components $D_1$ and $D_2$, then the Turaev surface $G(D)$ is the disjoint union $G(D_1)\coprod G(D_2)$, and hence $\chi(G(D))=\chi(G(D_1))+\chi(G(D_2))$. Also, if $D$ is a non-split alternating diagram, then $\chi(G(D))=2$.
\begin{figure}[h]
\includegraphics[scale=.7]{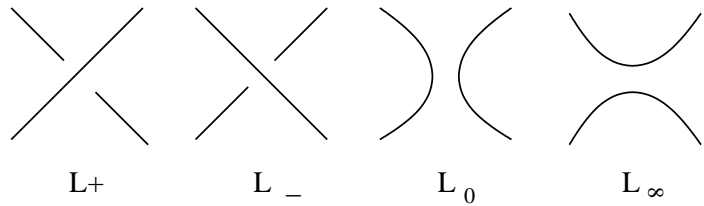}
\caption{The links in the skein relation.} \label{skein}
\end{figure}
\begin{theorem}
\label{skthm} Let $L_+,L_-,L_0,$ and $L_\infty$ be link diagrams as in Figure \ref{skein}. Then the following skein relation holds:
$$\chi(G(L_+))+ \chi(G(L_-)) = \chi(G(L_0))+ \chi(G(L_\infty))-2.$$
\end{theorem}
\begin{proof}

Let $a_i$ be the number of circles in the all $A$-splicing for $L_i$, $b_i$ be the number of circles in the all $B$-splicing for $L_i$, and $c_i$ the number of edges in either ribbon graph for $L_i$, where $i=+,-,0,$ or $\infty$. Since $L_0$ and $L_\infty$ have one less crossing than $L_+$ and $L_-$, it follows that $c_+=c_-=c_0+1=c_\infty+1$. Also since $L_0$ is obtained from $L_+$ by one $A$-splicing, the number of circles in the all $A$-splicing for $L_0$ is the same for $L_+$. Thus $a_+=a_0$. Similarly, $b_+=b_\infty,a_-=a_\infty$, and $b_-=b_0$. The result follows from these equations.
\end{proof}

The skein relation on Euler characteristic can be viewed as a relation on genus (and in light of Theorem \ref{width}, as a relation on width). Let $D$ be a link diagram. If $D$ is a non-split diagram, then let $\overline{g}(D)=g(G(D))$, the genus of the Turaev surface. However, if $D=D_1\coprod D_2$ is the disjoint union of two link diagrams $D_1$ and $D_2$, let $\overline{g}(D_1\coprod D_2)=\overline{g}(D_1)+\overline{g}(D_2)-1$. This normalization is introduced since $2-2\overline{g}(D_1\coprod D_2)=\chi(G(D_1\coprod D_2))$. Moreover, if $D$ is a non-split alternating diagram, then $\overline{g}(D)=0$. The skein relation of Theorem \ref{skthm} becomes
\begin{equation}
\overline{g}(L_+)+\overline{g}(L_-)=\overline{g}(L_0) + \overline{g}(L_\infty) +1. \label{genskein}
\end{equation}

Equation \ref{genskein} can also be viewed in terms of width. If $D$ is a non-split diagram, define $\overline{w}(D)=w(D)$, the width of the diagram. If $D=D_1\coprod D_2$ is a disjoint union of diagrams $D_1$ and $D_2$, introduce the normalization $\overline{w}(D_1\coprod D_2)=\overline{w}(D_1)+\overline{w}(D_2)-2$. Also, if $D$ is a non-split alternating diagram, then $\overline{w}(D)=1$. Theorem \ref{width} implies that the skein relation of Equation \ref{genskein} becomes
\begin{equation}
\overline{w}(L_+)+\overline{w}(L_-)=\overline{w}(L_0)+\overline{w}(L_\infty)+1. \label{widthskein}
\end{equation}

\noindent\textbf{Remarks:}
\begin{enumerate}
\item For most skein relations, the base case is a disjoint union of unknots; however, the base case of this skein relation is a disjoint union of alternating diagrams. 
\item Each of the checkerboard graphs for an alternating diagram is composed entirely of positive or negative edges. To transform an arbitrary link diagram into an alternating diagram, choose crossing changes that make all of the edges in one of the checkerboard graphs positive (and thus all of the edges in the other checkerboard graph negative).
\item If a link diagram is alternating, then for any crossing, the diagrams corresponding to $L_0$ and $L_\infty$ are also alternating.
\end{enumerate}

These remarks together imply that one can use Equation \ref{widthskein} to calculate the width of a diagram without computing the entire Kauffman state complex. We conclude with a simple example.

\begin{figure}[h]
\includegraphics[scale=.5]{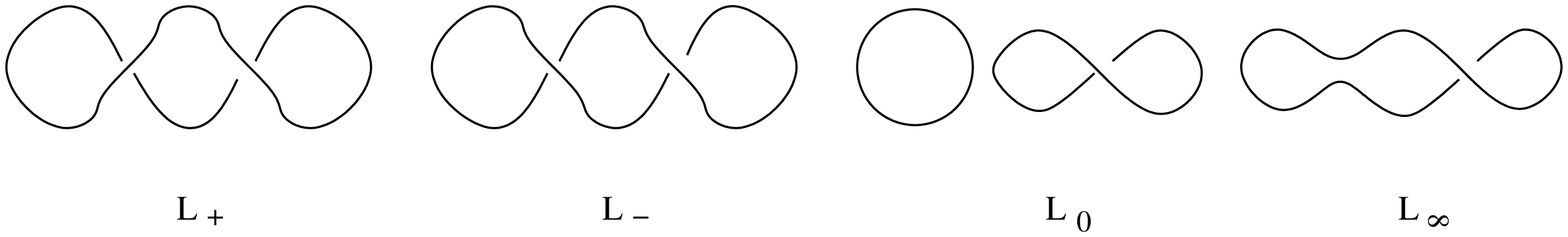}
\caption{Four link diagrams in the skein relation of Equation \ref{widthskein}.}
\end{figure}
Equation \ref{widthskein} states that 
$$\overline{w}(L_+)=-\overline{w}(L_-)+\overline{w}(L_0)+\overline{w}(L_\infty)+1.$$
Since $L_-$ and $L_\infty$ are non-split diagrams of alternating knots, it follows that $\overline{w}(L_-)=\overline{w}(L_\infty)=1$. However, $L_0=D_1\coprod D_2$ is a disjoint union of two non-split alternating diagrams $D_1$ and $D_2$. Therefore $\overline{w}(L_0)=\overline{w}(D_1)+\overline{w}(D_2)-2=0.$ Hence, $\overline{w}(L_+)=1$. This also follows directly from the Kauffman state complex since there is only one Kauffman state for the diagram $L_+$. However, in general, the number of Kauffman states increases exponentially with the number of crossings. Using the skein relation to calculate width depends only on the number of crossing changes needed to make $L_+$ into an alternating diagram.

\end{document}